\documentclass[12pt, a4paper]{article}

\usepackage{amsmath,amsthm,amssymb}
\usepackage{enumerate}

\theoremstyle{plain}
\newtheorem{theorem}{Theorem}[section]
\newtheorem{lemma}[theorem]{Lemma}

\theoremstyle{definition}
\newtheorem{definition}[theorem]{Definition}
\newtheorem{problem}[theorem]{Problem}

\newcommand{\s}{\mathcal}
\newcommand{\GL}{{\rm GL}}
\newcommand{\SL}{{\rm SL}}
\newcommand{\ppd}{{\rm ppd}}
\newcommand{\fat}{{\rm fat}}
\newcommand{\redfat}{\mbox{\rm{red}\hspace*{-.01cm}\emph{and}\hspace*{.02cm}\rm{fat}}}
\newcommand{\redIfFat}{\mbox{\rm{red}\emph{if}\rm{fat}}}

\begin{document}

\title{Irreducible linear subgroups generated by pairs of matrices with large irreducible submodules}
\author{Alice C. Niemeyer, Sabina B. Pannek, Cheryl E. Praeger}
\date{}

\maketitle

\begin{abstract}
We call an element of a finite general linear group $ \GL(d,q) $ \emph{fat} if it leaves invariant, and acts irreducibly on, a subspace of dimension greater than $d/2$. Fatness of an element can be decided efficiently in practice by testing whether its characteristic polynomial has an irreducible factor of degree greater than $d/2$. We show that for groups $G$ with $ \SL(d,q) \leq G \leq \GL(d,q) $ most pairs of fat elements from $G$ generate irreducible subgroups, namely we prove that the proportion of pairs of fat elements generating a reducible subgroup, in the set of all pairs in $ G \times G $, is less than $q^{-d+1}$.
We also prove that the conditional probability to obtain a pair $(g_1,g_2)$ in $G \times G$ which generates a reducible subgroup, given that $g_1, g_2$ are fat elements, is less than $2q^{-d+1}$.
Further, we show that any reducible subgroup generated by a pair of fat elements acts irreducibly on a subspace of dimension greater than $ d/2 $, and in the induced action the generating pair corresponds to a pair of fat elements.
\end{abstract}

\noindent\textbf{Mathematics Subject Classification (2010).} Primary 20G40; Secondary 20P05.

\noindent\textbf{Keywords.} General linear group, proportions of elements, large irreducible subspaces.

{\let\thefootnote\relax\footnote{{The results of this paper form part of the Australian Research Council funded project DP110101153 of the first and third author. The third author is also supported by the Australian Research Council Federation Fellowship FF0776186.\\
The second author is greatful for support of her PhD Fellowship funded by the German National Academic Foundation (Studienstiftung des deutschen Volkes). This paper is part of her PhD project as a co-tutelle student at RWTH Aachen University and the University of Western Australia.}}}

\section{Introduction}
Consider the finite general linear group $ \GL(d,q) $ for $ d \geq 3 $, that is the group of invertible $ (d \times d) $-matrices over the finite field $ \mathbb{F}_q $ of order~$ q $. For a subgroup $G$ of $ \GL(d,q) $ the underlying vector space of row vectors of length $ d $ over $ \mathbb{F}_q $ becomes a right $ \mathbb{F}_q G $-module via the natural ``vector times matrix'' action. We call this module the \emph{natural $ \mathbb{F}_q G $-module}. An element $ g \in \GL(d,q) $ is said to be \emph{fat}, or more precisely a \emph{$\fat(d,q;e)$-element}, if the natural $ \mathbb{F}_q \GL(d,q) $-module has an irreducible $ \mathbb{F}_q\langle g \rangle $-submodule of dimension $e > d / 2 $, or equivalently, if the characteristic polynomial for $g$ has an irreducible factor over $\mathbb{F}_q$ of degree $e$. \emph{Fat pairs}, that is pairs of fat elements, relative to the (not necessarily distinct) integers $ e_1,e_2$ are called \emph{$\fat(d,q;e_1,e_2)$-pairs}. Further, a pair $ (g_1,g_2)$ in $\GL(d,q) \times \GL(d,q) $ is said to be \emph{reducible} or \emph{irreducible} according as the natural $ \mathbb{F}_q \langle g_1,g_2\rangle $-module has this property.

Let $ \SL(d,q) $ denote the finite special linear group, the group of all matrices in $ \GL(d,q) $ with determinant 1. Motivated by the wish to upgrade the Classical Recognition Algorithm \cite{MR1625479} (see discussion in Section~$\ref{section motivation}$), we study fat pairs in $G \times G$ for a matrix group $G$ satisfying $ \SL(d,q) \leq G \leq \GL(d,q) $. We first give an explicit upper bound for the proportion of reducible fat pairs in the set of all pairs in $G \times G $. We denote this proportion by $\redfat(G)$.

\begin{theorem}\label{theorem proportion ref and fat}
	Let $d \geq 3$. If $G$ is a group with $\SL(d,q)\leq G\leq\GL(d,q) $, then
		\[\redfat(G)  < q^{-d+1}.\]
\end{theorem}

Let $\redIfFat(G)$ be the proportion of reducible pairs in the set of fat pairs in $G \times G$. Equivalently, we may define $\redIfFat(G)$ to be the (conditional) probability that, on a single random selection from the set of fat pairs in $G \times G$, we obtain a reducible pair. An upper bound for $\redIfFat(G)$ is given~in

\begin{theorem}\label{theorem proportion red if fat}
	Let $d \geq 3$. If $G$ is a group with $\SL(d,q)\leq G\leq\GL(d,q) $, then
		\[\redIfFat(G)  < 2q^{-d+1}.\]
\end{theorem}

Our next theorem shows that each reducible fat pair leads to an irreducible fat pair on a quotient space of dimension greater than $d/2$.

\begin{theorem}\label{theorem reduction}
	For integers $ d,e_1,e_2 $ satisfying $ 1 < d/2 < e_1,e_2 \leq d $, let $ (g_1, g_2) \in \GL(d,q) \times \GL(d,q) $ be a $ \fat(d,q;e_1,e_2) $-pair, and let $ \s{V} $ be the natural $ \mathbb{F}_q \GL(d,q)$-module. Then there exists an $ \mathbb{F}_q \langle g_1, g_2 \rangle $-composition factor $ \s{N}$ of~$\s{V} $ with $ n= \dim(\s{N}) \geq \max\{e_1,e_2\} > d/2$, such that writing $ \overline{g_i} $ for the element in $ \GL(n,q) $ induced by $ g_i $ on $ \s{N} $, $(\overline{g_1}, \overline{g_2}) $ is an irreducible $ \fat(n,q; e_1, e_2) $-pair.
\end{theorem}

The proofs of Theorems~$\ref{theorem proportion ref and fat}$ and~$\ref{theorem proportion red if fat}$ (see Subsections~$\ref{subsection proof of Theorem red and fat}$ and~$\ref{subsection theorem proportion red if fat}$) rely on the following observation. A fat pair $(g_1,g_2) \in G \times G$, where $G$ satisfies $\SL(d,q) \leq G \leq \GL(d,q)$, is reducible, if and only if there exists a non-trivial and proper $\langle g_1, g_2 \rangle$-invariant subspace $\s{W} \leq \s{V}$. In this case $g_1,g_2$ lie in the maximal parabolic subgroup $G_\s{W} \leq G$. The key ingredient to prove Theorems~$\ref{theorem proportion ref and fat}$ and~$\ref{theorem proportion red if fat}$ is to show in Lemma~$\ref{lemma fat(G,e) = fat(e) = ...}$ that for $e > d/2$ the proportion of $\fat(d,q;e)$-elements in $G_\s{W}$ equals the proportion of $\fat(d,q;e)$-elements in $\GL(d,q)$. The results then follow by summing the number of fat pairs over all possible maximal parabolic subgroups of $G$. The proof of Theorem~$\ref{theorem reduction}$ is presented in Subsection~$\ref{subsection proof of Theorem reduction}$. In Section~$\ref{section motivation}$ we motivate the results of this paper. The linear algebra background required is presented in Section~$\ref{section linear algebra preliminaries}$, while the group theoretic preliminaries are in Section~$\ref{section group theory preliminaries}$.

% ------------------------------------------------------------------------

\section{Motivation}\label{section motivation}

The principal motivation for the work reported in this paper is the \emph{Classical  Recognition Algorithm} \cite{MR1625479}. This is a one-sided Monte Carlo algorithm that, given a set of generating matrices for a subgroup $G$ of the finite general linear group $\GL(d,q)$, examines whether $G$ contains a ``classical group'' in its natural representation, that is whether (in its natural representation) $G$ contains $\SL(d,q)$, or a $d$-dimensional symplectic, unitary or orthogonal group defined over $\mathbb{F}_q$. The performance of the algorithm has been described by Leedham-Green in \cite{MR1829483} as ``one of the most efficient algorithms in the business''. The algorithm seeks particular kinds of elements, called ppd-elements, in $G$ by making independent uniformly distributed random selections of elements from $G$. A \emph{ppd-element}, or more precisely a \emph{$\ppd(d,q;e)$-element} for some integer $e$ with $e \leq d$, is an element $ g \in \GL(d, q)$ such that $g$ has order divisible by a prime divisor of $q^e-1$ which does not divide $q^j - 1$ for any $j < e$. It is shown in \cite{MR1625479} that $\ppd(d,q;e)$-elements with $ e $ greater than $ d/2 $ are very likely to occur in classical groups. Under some additional hypotheses, finding a pair of ppd-elements from $G$ allows us to conclude that $G$ contains a classical group. The proof of this relies on good estimates of the proportions of ppd-elements along with deep group theoretic analysis (depending on the simple group classification). In the long run, we wish to upgrade the Classical Recognition Algorithm in a threefold manner as described below. This paper takes a first step in this direction.

First, note that by \cite[Lemma 5.1]{MR1625479}, given a $\ppd(d,q;e)$-element $g$ with $e>d/2$, there exists a unique irreducible $e$-dimensional $ \mathbb{F}_q\langle g \rangle $-submodule of the natural $\mathbb{F}_q \GL(d,q)$-module. In particular, $ g $ is a $\fat(d,q;e)$-element. While every ppd-element is fat the converse implication is not true, as the presence of an $ e$-dimensional irreducible $\mathbb{F}_q\langle g\rangle$-submodule of the natural $\mathbb{F}_q \GL(d,q)$-module is not sufficient to guarantee  that $g$ is a $\ppd(d,q;e)$-element. For example in $\GL(3,3)$, an element of order $8$ is a $ \fat(3,3;2) $-element but not a  $\ppd(3,3;2)$-element since $3^2-1 = 8$ has no prime divisors which do not divide $ 3-1 =2$. However, even though fat elements do not necessarily need to be ppd-elements, most of them turn out to be. Our goal is to remove the restriction of looking for ppd-elements in the Classical Recognition Algorithm and evolve the algorithm into one based solely on elements with large irreducible submodules. Dropping the \ppd-property should result in an even better performance of the algorithm as in practice fatness can be tested more cheaply than the $\ppd$-property by finding an irreducible factor of degree greater than $d/2$ of the characteristic polynomial. The wish
to waive the ppd-property raises the following problem which we intend to address in further work.

\begin{problem}
	Describe all subgroups of $\GL(d,q) $ containing an irreducible $ \fat(d,q;e_1,e_2) $-pair for $ 1 < d/2 < e_1, e_2 \leq d $.
\end{problem}

As presented in \cite{MR1625479}, the Classical Recognition Algorithm takes as input a basis for the non-degenerate sesquilinear forms preserved by the subgroup $ G\leq \GL(d,q) $, as well as the knowledge that $G$ is irreducible on the underlying vector space. This requirement  is reasonable as efficient algorithms for testing irreducibility exist (namely the Meataxe algorithm due to Richard Parker \cite{MR760660} and the improved, general purpose version of it developed by Holt and Rees \cite{MR1279282}). Yet, we wish to develop a new (fat element based) recognition algorithm without the necessity to test for irreducibility. In order to evaluate how this move modifies the situation, Theorem~$\ref{theorem proportion red if fat}$ gives a good upper bound for the (conditional) probability of obtaining, on a single random selection from the set of fat pairs in $G \times G$ (where $\SL(d,q) \leq G \leq \GL(d,q)$), a reducible pair. We expect that similar bounds will hold if $\Omega(d,q) \leq G \leq N_{\GL(d,q)}(\Omega(d,q)) $ for any classical group $ \Omega(d,q) $.

Finally, by Theorem~$\ref{theorem reduction}$, if for a given matrix group $G\leq\GL(d,q)$ with $d\geq 3$, $G \times G $ contains a fat pair, then $G$ has a quotient $ H $ which is isomorphic to a matrix group of degree $ n > d/2$, such that $H \times H$ contains an irreducible fat pair. This suggests that recognition of groups containing classical groups could be generalised to test if a (reducible) subgroup of $\GL(d,q)$ has a large quotient containing $\SL(n, q)$ or an $ n $-dimensional symplectic, unitary or orthogonal group, with $n > d/2$.

% ------------------------------------------------------------------------

\section{Linear algebra preliminaries}\label{section linear algebra preliminaries}

Throughout this section let $q $ be a power of a prime, $d$ a non-negative integer, and $\s{V}$ a $d$-dimensional vector space defined over the finite field $\mathbb{F}_q$.

The proofs of Theorems~$\ref{theorem proportion ref and fat}$ and~$\ref{theorem proportion red if fat}$ involve counting certain subspaces in~$\s{V}$. As usual we denote the number of $w$-dimensional subspaces in $\s{V}$ (for $0 \leq w \leq d$) by so-called Gaussian coefficients (see for example \cite[p. $124$]{MR1311922}).

\begin{definition}\label{defintion gaussian coefficient}
	For a non-negative integer $w \leq d$, the \emph{Gaussian coefficient} $\binom{d}{w}_q$ is defined to be the number of $w$-dimensional subspaces in $\s{V}$.
\end{definition}

An explicit formula for $\binom{d}{w}_q$ is given for example in \cite[$(9.2.2)$]{MR1311922}.

\begin{lemma}\label{lemma gaussian coefficient}
	Let $w\leq d$ be a non-negative integer. Then
		\[\binom{d}{w}_q = \frac{ \prod_{i=d-w+1}^d (q^i - 1) }{ \prod_{i=1}^w (q^i-1)},\]
	and in particular $ \binom{d}{w}_q = \binom{d}{d-w}_q.$
\end{lemma}
	
For a rational number $r$ let $\lceil r \rceil$ be the smallest integer which is at least~$r$.

\begin{lemma}\label{lemma sum over gaussian coeff inverse}
	If $d \geq 3$, then
		$\sum_{i=1}^{\lceil d/2 \rceil -1} \binom{d}{i}_q^{-1} < q^{-d+1}.$
\end{lemma}

\begin{proof}
	Since for $i \in [1,d]$, $\binom{d}{i}_q$ is the number if $i$-dimensional subspaces in $\s{V}$, we have $ \binom{d}{2}_q < \binom{d}{i}_q $ for $ 2 < i \leq \lceil d/2 \rceil -1 $, and obtain
	\[\sum_{i=1}^{\lceil d/2 \rceil -1} q^{d-1} / \textstyle \binom{d}{i}_q 
			\leq q^{d-1} / \binom{d}{1}_q + (\lceil d/2 \rceil -2) q^{d-1} / \binom{d}{2}_q.\]
 Note that $q^{d-1} /\binom{d}{1}_q < 1-q^{-1} +q^{-d}$ and $q^{d-1}/\binom{d}{2}_q < q^{-d+3}$, whence
	\[\sum_{i=1}^{\lceil d/2 \rceil -1} q^{d-1}/ \textstyle\binom{d}{i}_q < 1-q^{-1}+q^{-d} + (\lceil d/2 \rceil -2) q^{-d+3} =: \mu(d,q).\] If $ d \in \{3,4\} $, then $ \mu(d,q) < 1 $. For $ d\geq 5 $ we use induction on $ d $ to show that $ \mu(d,q) < 1 $. Now, $ \mu(5,q) = 1-q^{-1}+q^{-5}+q^{-2} < 1$. Next, assuming $ \mu(d,q) < 1 $, we have
 	\begin{align*}
 		\mu(d+1,q)
 				& =  1 -q^{-1} + q^{-d-1} + (\lceil (d+1)/2 \rceil -2) q^{-d+2} \\
  				& < 1 -q^{-1} +q^{-d} + (\lceil d/2 \rceil - 1) q^{-d+2} +\\
  				& \quad  + \underbrace{(\lceil d/2 \rceil - 2) q^{-d+3} - (\lceil d/2 \rceil - 2) q^{-d+3}}_{=0}.
 	\end{align*}
By assumption, $ \mu(d,q) = 1 -q^{-1} +q^{-d} + (\lceil d/2 \rceil - 2) q^{-d+3} < 1 $, and thus
	\[\mu(d+1,q) < 1 + (\lceil d/2 \rceil - 1)q^{-d+2} - (\lceil d/2 \rceil - 2)q^{-d+3}.\]
Using $q \geq 2$ and $d \geq 5$,
 	\begin{align*}
 		\mu(d+1,q) 			
 			& < 1 + (\lceil d/2 \rceil - 1)q^{-d+2} - 2(\lceil d/2 \rceil - 2)q^{-d+2}
			\\ & = 1 - q^{-d+2} \left( \lceil d/2 \rceil -3\right)
 			\leq 1.
 	\end{align*}
We therefore have $\sum_{i=1}^{\lceil d/2 \rceil -1} q^{d-1} / \binom{d}{i}_q < \mu(d,q) < 1$, as asserted.
\end{proof}

% ------------------------------------------------------------------------

\section{Group theory preliminaries}\label{section group theory preliminaries}

In this section we assume that $d \geq 2$ is an integer, and $q$ is a power of a prime. Let $\s{V}$ be the natural $\mathbb{F}_q \GL(d,q)$-module, that is the vector space of $d$-dimensional row vectors over $\mathbb{F}_q$ on which $\GL(d, q)$ acts naturally.

For $G \leq \GL(d,q)$ and a subspace $\s{W} \leq \s{V}$, we denote by $G_\s{W}$ the subgroup of $G$ which leaves $\s{W}$ invariant, that is $G_\s{W} = \{g \in G \mid \s{W}g = \s{W} \}$.
Using an argument very similar to \cite[proof of Theorem 4.1]{MR1189139} we obtain

\begin{lemma}\label{lemma G_W is transitive on U}
	Let $e, w \in [0,d]$ be integers, and $\s{W} \leq \s{V}$ of dimension~$w$.
	\begin{enumerate}[$(a)$]
		\item If $e +w \leq d$, then $\SL(d,q)_\s{W}$ acts transitively on the set of all $e$-dimensional subspaces $\s{U} \leq \s{V}$ such that $\s{U} \cap \s{W} = \{0\}$.
		\item If $e \leq w \leq d$, then $\SL(d,q)_\s{W}$ acts transitively on the set of all $e$-dimensional subspaces $\s{U} \leq \s{W}$.
	\end{enumerate}
	In particular, $SL(d,q)$ is transitive on the all $e$-dimensional subspaces in~$\s{V}$.
\end{lemma}

As specified in the introduction, we call an element $g \in \GL(d,q)$ a $\fat(d,q;e)$-element, if $\s{V}$ has an irreducible $\mathbb{F}_q \langle g \rangle$-submodule of dimension $e > d/2$. In the remainder of this section we shall be concerned with the proportions of $\fat(d,q;e)$-elements in (maximal parabolic subgroups of) $G$, where $G$ satisfies $\SL(d,q) \leq G \leq \GL(d,q)$.

\begin{definition}\label{definition proportion of fat elements}
	For an integer $e \in (d/2,d]$ and $G \leq \GL(d,q)$, define $\fat(G;e)$ to be the proportion of $\fat(d,q;e)$-elements in~$G$. Set $\fat(e) := \fat(\GL(e,q);e)$.	
\end{definition}

\begin{lemma}\label{lemma proportion of fat elements}
	For an integer $e \geq 2$ we have $1/(e+1) \leq \fat(e) < 1/e$.
\end{lemma}

\begin{proof}
	For $e \geq 3$, the lower bound is given in \cite[Lemma 2.3]{MR1182102}. From the proof of the same lemma it follows that for all $e \geq 2$ we have $\fat(e) = \vert C_0 \vert / (e \vert C \vert)$, where $C_0$ is a proper subset of $C \leq \GL(e,q)$ with
		\[\vert C \vert = q^e-1, \quad \vert C \vert - \sum_{f \mid e \text{ proper}} (q^f-1) \leq \vert C_0 \vert < \vert C \vert.\]
	For $e = 2$ we thus get (using $q \geq 2$)
		\[\fat(2) = \frac{\vert C_0 \vert}{e \vert C \vert} \geq \frac{q^2-1 - (q-1)}{2 (q^2 - 1)} = \frac{q}{2(q+1)} \geq \frac{2}{2 \cdot 3} = \frac{1}{3},\]
	as required. Since $\vert C_0 \vert / \vert C \vert < 1$, the upper bound follows (for all $e \geq 2$).	
\end{proof}

The proof of the following lemma is based on \cite[proof of Lemma 5.4]{MR1625479}.

\begin{lemma}\label{lemma fat(G,e) = fat(e) = ...}
	Let $G$ be a group satisfying $\SL(d,q) \leq G \leq \GL(d,q)$, and let $\s{W} \leq \s{V}$. Let $e$ be an integer such that $e \in (d/2,d)$. Then, $G_\s{W}$ contains a $\fat(d,q;e)$-element if and only if $\dim(\s{W}) \in [0,d-e] \cup [e,d]$, and in this case
		\[\fat(G_\s{W}; e) = \fat(e).\]
	In particular, $\fat(G;e) = \fat(e)$.
\end{lemma}

\begin{proof}
	We set $H := G_\s{W}$. If $\dim(\s{W}) \in [0,d-e] \cup [e,d]$, then it is easy to verify that $H$ contains a $\fat(d,q;e)$-element.

	Conversely, suppose that $H$ contains a $\fat(d,q;e)$-element $g$, and let $\s{U}$ be the irreducible $\mathbb{F}_q \langle g\rangle$-submodule of $\s{V}$ with $\dim(\s{U}) = e$. 	Note that $\s{U}$ is uniquely determined, as it is irreducible and of dimension $e > d/2$. The intersection $\s{U} \cap \s{W}$ is an $\mathbb{F}_q\langle g \rangle$-submodule of $\s{U}$. Hence $\s{U} \cap \s{W} \in \{\{0\},\s{U}\}$, and in particular $\dim(\s{W}) \leq d-e$ or $\dim(\s{W}) \geq e$.
	Recall from Lemma~$\ref{lemma G_W is transitive on U}$ that $H$ acts transitively on the set $\mathbf{U}$, where
		\[\mathbf{U} := \begin{cases}
										\{\s{U}' \leq \s{V} \mid \dim(\s{U}') = e, \, \s{U}' \cap \s{W} = \{0\} \}, & \text{if } \dim(\s{W}) \leq d-e,\\
										\{\s{U}' \leq \s{V} \mid \dim(\s{U}') = e, \, \s{U}' \leq \s{W} \}, & \text{if } \dim(\s{W}) \geq e. 
									\end{cases}\]
	Since $\s{U} \in \mathbf{U}$, by the orbit stabiliser theorem $\vert \mathbf{U} \vert = \vert H : H_\s{U} \vert$. Thus, the number of $\fat(d,q;e)$-elements in $H$ equals $\vert H : H_\s{U} \vert$ times the number of $\fat(d,q;e)$-elements in $H_\s{U}$, that is
	$\fat(H;e) \vert H \vert =  \vert H : H_\s{U} \vert \fat( H_\s{U}; e ) \vert H_\s{U} \vert$, whence $\fat(H;e) = \fat( H_\s{U}; e)$.
	
	Let $\mathfrak{X}: H_\s{U} \rightarrow \GL(e,q)$ be the representation afforded by $\s{U}$ as an $\mathbb{F}_q H_\s{U}$-submodule of~$\s{V}$. Let $ \ker(\mathfrak{X}) $ be the kernel of $ \mathfrak{X}$. If for $ g \in H_\s{U} $ the coset $\ker(\mathfrak{X})g $ contains a $ \fat(d,q; e) $-element, then every element of $ \ker(\mathfrak{X})g $ is a $ \fat(d,q; e) $-element.
	It follows that the number of $ \fat(d,q;e) $-elements in $H_\s{U}$ equals $ \vert \ker(\mathfrak{X}) \vert $ times the number of $ \fat(e,q;e) $-elements in $ \mathfrak{X}(H_\s{U} )$, that is
	$\fat(H_\s{U}; e) \vert  H_\s{U} \vert = \vert \ker(\mathfrak{X}) \vert \fat( \mathfrak{X}( H_\s{U}) ;e) \vert \mathfrak{X}( H_\s{U} ) \vert$.
	Then, using $\vert H_\s{U} \vert = \vert \ker(\mathfrak{X}) \vert \vert \mathfrak{X}(H_\s{U}) \vert$, we get $\fat(H_\s{U}; e) = \fat(\mathfrak{X}(H_\s{U});e) $.
	
	Finally, since $e < d$, we have $\mathfrak{X} (H_\s{U}) \cong \GL(e,q),$ and thus $\fat(H_\s{U}; e) = \fat(e)$.
	This proves the assertion, as $\fat(H;e) = \fat(H_\s{U};e) = \fat(e)$.
	
	By setting $\s{W} := \{0\}$ we obtain that $\fat(G;e) = \fat(e)$.
\end{proof}

% ------------------------------------------------------------------------

\section{Proofs of main results}

Throughout this section let $d \geq 3$ be a positive integer, $\mathbb{F}_{q}$ a finite field of order $q$ for some prime power $q$, and $ \s{V} $ the natural $ \mathbb{F}_q \GL(d,q) $-module.

\subsection{Proof of Theorem~\ref{theorem reduction}}\label{subsection proof of Theorem reduction}

If $(g_1,g_2) \in \GL(d,q) \times \GL(d,q) $ is a $\fat(d,q;e_1,e_2)$-pair for some integers $e_1,e_2 > d/2$,
%that is for $ i = 1,2 $, $ g_i $ is a $ \fat(d,q;e_i)$-element,
then by definition $ g_i $ determines an (uniquely determined) $ e_i$-dimensional irreducible $ \mathbb{F}_q \langle g_i \rangle $-submodule $\s{U}_i$ of $\s{V} $ $(i=1,2).$ 
%Recall that $ \s{U}_i $ is uniquely determined by $ g_i $ for $ i = 1,2 $, since $ \s{U}_i $ is irreducible and $ \dim(\s{U}_i) > d/2 $.
In addition, there may or may not exist a proper and non-trivial $\mathbb{F}_q \langle g_1, g_2 \rangle $-submodule $ \s{W} $ of $ \s{V}$ according as $ (g_1, g_2) $ is reducible or not. The following lemma presents a basic, yet critical property of $\s{W} $ in such a setting. Note that, if $ \max\{e_1,e_2\} = d$, then $(g_1, g_2)$ is irreducible. Hence, in order that $\s{W}$ exists, we assume that each $e_i < d$. We write $\langle \s{U}_1,\s{U}_2\rangle_{\mathbb{F}_q \langle g_1, g_2\rangle}$ for the intersection of all $\mathbb{F}_q \langle g_1,g_2 \rangle$-submodules in $\s{V}$ which contain $\s{U}_1$ and $\s{U}_2$.

\begin{lemma}\label{lemma cases for W}
	Let $ e_1, e_2 \in \mathbb{N}$ with $1<d/2 < e_1,e_2 < d $, and let $ (g_1, g_2) $ be a reducible $ \fat(d,q; e_1,e_2) $-pair in $ \GL(d,q) \times \GL(d,q)$. For $ i = 1,2 $ let $ \s{U}_i $ denote the irreducible $ \mathbb{F}_q \langle g_i \rangle $-submodule of $ \s{V} $ of dimension $ e_i $, and let $ \s{W} \not\in \{ \{0\}, \s{V}\} $ be a $ \mathbb{F}_q \langle g_1, g_2 \rangle $-submodule of $ \s{V} $. Then exactly one of the following holds:
	\begin{enumerate}[$(a)$]
		\item $ \s{W} \cap \s{U}_i = \{ 0 \} $, and $ 1 \leq \dim(\s{W}) \leq d - \max\{e_1, e_2\} $, or
		\item $\langle \s{U}_1,\s{U}_2\rangle_{\mathbb{F}_q \langle g_1, g_2\rangle} \leq \s{W}$ and
		$ \max\{e_1, e_2\} \leq \dim(\s{W}) \leq d-1 $.
	\end{enumerate}	
	In particular, $\dim(\s{W}) \in [1, d - \max\{e_1, e_2\}] \cup [\max\{e_1, e_2\}, d-1] $.
\end{lemma}

\begin{proof}
	For $ i = 1,2 $ the intersection $\s{W} \cap \,\s{U}_i$ is an $ \mathbb{F}_q \langle g_i \rangle $-submodule of $ \s{U}_i $. Since $ \s{U}_i $ is irreducible it follows that $\s{W}\cap\s{U}_i$ is trivial or non-proper. Suppose that for some $ i \in \{1,2\} $, $ \s{W} \cap \s{U}_i = \{0\} $ and $ \s{W} \cap \s{U}_{3-i} = \s{U}_{3-i} $.	 Then $\s{U}_1\cap\s{U}_2=\{0\}$ which contradicts $\dim(\s{U}_i) = e_i > d/2$. Thus either $ \s{W} \cap \s{U}_i = \{0\} $ for $ i = 1,2 $, or $ \s{W} \cap \s{U}_i = \s{U}_i $ for $ i = 1,2 $. In the first case, $ 1 \leq \dim(\s{W}) \leq d - \max\{e_1,e_2\} $ and (a) holds. In the second case, $ \max\{e_1,e_2\} \leq \dim(\s{W}) \leq d-1 $, and as each $ \s{U}_i \leq \s{W}$, also $\langle \s{U}_1,\s{U}_2\rangle_{\mathbb{F}_q \langle g_1, g_2 \rangle} \leq~\s{W}$, so (b) holds.
\end{proof}

\begin{proof}[Proof of Theorem~$\ref{theorem reduction}$]	
	For $i=1,2$, let $\s{U}_i$ denote the $\mathbb{F}_q \langle g_i \rangle$-submodule of $\s{V}$ with $\dim(\s{U}_i)= e_i$. Let $ \s{X} :=\langle \s{U}_1, \s{U}_2 \rangle_{\mathbb{F}_q \langle g_1, g_2 \rangle} $, and let $\s{Y} $ be an $ \mathbb{F}_q \langle g_1, g_2\rangle $-submodule of $\s{X}$ maximal by inclusion with respect to the property $ \s{U}_1 \cap \s{Y} = \s{U}_2 \cap \s{Y} = \{0\} $. Define $ \s{N} = \s{X}/\s{Y}$.
	For $ i=1,2 $, $ \s{U}_i \cong (\s{U}_i \oplus \s{Y}) \bigl/ \s{Y} \leq \s{X} / \s{Y} $ can be viewed as a submodule of $ \s{N} $. It follows that $\dim(\s{N}) \geq \max\{e_1, e_2\}$, and that the pair $ (\overline{g_1}, \overline{g_2})$ induced by $ (g_1, g_2) $ on $ \s{N} \times \s{N} $ is a $ \fat(n,q;e_1,e_2) $-pair.
    
	It remains to prove that $ (\overline{g_1}, \overline{g_2})$ is irreducible, that is $ \s{N} $ is an $ \mathbb{F}_q \langle g_1, g_2\rangle $-composition factor of $ \s{V} $. We do this by showing that $\s{Y} $ is a maximal $\mathbb{F}_q \langle g_1, g_2 \rangle $-submodule of $ \s{X}$. Suppose that there exists an $\mathbb{F}_q \langle g_1, g_2\rangle $-module $\s{W} $ satisfying $\s{Y} < \s{W} < \s{X}$. By Lemma~$\ref{lemma cases for W}$, we either have $ \s{W}\cap \s{U}_i
	  = \{0\}$ for $ i = 1,2 $, or $ \langle \s{U}_1, \s{U}_2\rangle_{\mathbb{F}_q \langle g_1, g_2 \rangle} \leq \s{W}$. Since $ \s{X} = \langle \s{U}_1, \s{U}_2 \rangle_{\mathbb{F}_q \langle g_1, g_2 \rangle} $ and $\s{X} \not\leq \s{W} $, the latter case cannot occur. Hence, $ \s{W} $ is a proper $ \mathbb{F}_q \langle g_1, g_2\rangle $-submodule of $ \s{X} $ that satisfies $ \s{W}\cap\s{U}_i = \{0\} $ and properly contains $\s{Y} $. This, however, is not true as we have chosen $ \s{Y} $ to be maximal with respect to this property.
\end{proof}

%-----------------------------------------------------------------------

\subsection{Proof of Theorem~\ref{theorem proportion ref and fat}}\label{subsection proof of Theorem red and fat}

Given a group~$ G $, which satisfies $ \SL(d,q) \leq G \leq \GL(d,q) $, we wish to find a good upper bound for the proportion $ \redfat(G) $ of reducible fat pairs in $ G \times G$. As a first step, we consider the proportion of reducible fat pairs relative to some fixed parameters $e_1,e_2 >d/2$.

\begin{definition}\label{definition redfat(G,e1,e2)}
	For a group $G$ such that $\SL(d,q) \leq G \leq \GL(d,q) $, and integers $e_1, e_2 \in (d/2,d]$ we define $\redfat(G;e_1,e_2)$ to be the proportion of reducible $\fat(d,q;e_1,e_2)$-pairs in the set of all pairs in $G \times G$.
\end{definition}
 
\begin{lemma}\label{lemma proportion}
	Let $ e_1,e_2 \in \mathbb{N} $ such that $d/2 < e_1,e_2 < d $, and let $ G $ be a group satisfying $\SL(d,q)\leq G\leq\GL(d,q) $. Then
	\begin{align*}
		\redfat(G;e_1,e_2) < 2 \, \fat(e_1) \, \fat(e_2) \, q^{-d+1}  < 2/(e_1e_2)q^{-d+1}. %\leq \frac{8}{(d+1)^2} q^{-d+1}.
	\end{align*}
\end{lemma}

\begin{proof}
	The pair $(g_1,g_2) \in G \times G$ is a reducible $\fat(d,q;e_1,e_2)$-pair if and only if there exists at least one non-trivial and proper subspace $\s{W} \leq \s{V}$ such that $g_i$ is a $\fat(d,q;e_i)$-element in $G_\s{W} $. By Lemma~$\ref{lemma cases for W}$, $\dim(\s{W}) \in [1,d-\max\{e_1,e_2\}] \cup [\max\{e_1,e_2\},d-1]$. We thus obtain the following upper bound for the number of reducible $\fat(d,q;e_1,e_2)$-pairs in $G \times G$.
		\[\redfat(G;e_1,e_2)\vert G \vert^2 \leq \sum_{w} \sum_\s{W} \prod_{i=1,2} ( \fat(G_\s{W};e_i) \, \vert G_\s{W} \vert ),\]
	where $w \in [1,d-\max\{e_1,e_2\}] \cup [\max\{e_1,e_2\},d-1]$, and $\s{W} \leq \s{V}$ with $\dim(\s{W}) = w$. By Lemma~$\ref{lemma fat(G,e) = fat(e) = ...}$, $\prod_{i=1,2} \fat(G_\s{W};e_i) = \fat(e_i)$, and hence
		\[\redfat(G;e_1,e_2) \leq \fat(e_1)\fat(e_2) \sum_{w} \sum_\s{W} \vert G:  G_\s{W} \vert^{-2},\]
	with $w, \s{W}$ as before. Since $G$ acts transitively on the set of all $w$-dimensional subspaces in $\s{V}$ there is a total of $\vert G: G_\s{W} \vert$ such subspaces, whence
		\[\redfat(G;e_1,e_2) \leq \fat(e_1)\fat(e_2)\sum_{w} \vert G:  G_\s{W} \vert^{-1},\]
	where $w$ as before. Using the notation from Definition~$\ref{defintion gaussian coefficient}$, we write $\binom{d}{w}_q = \vert G:G_\s{W} \vert$.
	 Then, since $\binom{d}{w}_q = \binom{d}{d-w}_q $, and since $ d- \max\{e_1, e_2\} \leq \lceil d/2 \rceil -1 $,
    \[\redfat(G;e_1,e_2) \leq 2 \, \fat(e_1) \, \fat(e_2)  \sum_{w=1}^{\lceil d/2 \rceil -1} \textstyle \binom{d}{w}_q^{-1}.\]
	Then, by Lemmas~$\ref{lemma sum over gaussian coeff inverse}$ and~$\ref{lemma proportion of fat elements}$, $\redfat(G;e_1,e_2) < 2 \fat(e_1) \fat(e_2) q^{-d+1} < 2/(e_1e_2)q^{-d+1}$.
\end{proof}

Note that for $G$ with $\SL(d,q) \leq G \leq \GL(d,q)$ we have $\redfat(G) = \sum_{d/2 < e_1,e_2 \leq d} \redfat(G;e_1,e_2)$.
This observation together with the upper bound given in Lemma~$\ref{lemma proportion}$ are the main ingredients of the

\begin{proof}[Proof of Theorem~$\ref{theorem proportion ref and fat}$]
	In the case $d = \max\{e_1,e_2\}$ any $\fat(d,q;e_1,e_2)$-pair is irreducible, and thus $\redfat(G;e_1,e_1) = 0$. Hence, using Lemma~$\ref{lemma proportion}$,
	\begin{align*}
		\redfat(G)	 = \sum_{e_1,e_2} \redfat(G;e_1,e_2) < \sum_{e_1,e_2} 2 /(e_1e_2) q^{-d+1},
	\end{align*}
 where $\lceil (d+1)/2 \rceil \leq e_1, e_2 \leq d-1$.
 An easy argument estimating the sum by an integral shows that $\sum_{i = \lceil (d+1)/2 \rceil}^{d-1} i^{-1} < \ln(2)$.
Hence, $\redfat(G) < 2 \bigl( \ln(2) \bigr)^2 q^{-d+1} < q^{-d+1}$, as required.
\end{proof}

\subsection{Proof of Theorem~\ref{theorem proportion red if fat}}\label{subsection theorem proportion red if fat}

Recall that for a group $G$ with $\SL(d,q) \leq G \leq \GL(d,q)$ we write $\redIfFat(G)$ for the proportion of reducible fat pairs in the set of fat pairs from $G \times G$. Our final task is to prove the upper bound for $\redIfFat(G)$ given in Theorem~$\ref{theorem proportion red if fat}$.

\begin{proof}[Proof of Theorem~$\ref{theorem proportion red if fat}$]
	For integers $e_1,e_2$ with $d/2 < e_1,e_2 \leq d$ we write $\redfat(G;e_1,e_2) \, \vert G \vert^2$ for the number of reducible $\fat(d,q;e_1,e_2)$-pairs in $G \times G$, and $\fat(G;e_1) \fat(G;e_2) \, \vert G \vert^2$ for the number of $\fat(d,q;e_1,e_2)$-pairs in $G \times G$. By Lemma~$\ref{lemma fat(G,e) = fat(e) = ...}$ we have $\fat(G;e_i) = \fat(e_i)$ for $i = 1, 2$, whence
		\[\redIfFat(G) = \frac{\sum_{d/2 < e_1,e_2 \leq d} \redfat(G;e_1,e_2) \, \vert G \vert^2}{\sum_{d/2 < e_1,e_2 \leq d} \fat(e_1) \, \fat(e_2)\, \vert G \vert^2}.\]
	If $\max\{e_1,e_2\} = d$, then every $\fat(d,q;e_1,e_2)$-pair in $G$ is irreducible, and hence $\redfat(G;e_1,e_2) = 0$ in that case. If $e_1,e_2 \in (d/2,d)$, then
		\[\redfat(G;e_1,e_2) < 2 \, \fat(e_1)\, \fat(e_2) \, q^{-d+1}\]
	by Lemma~$\ref{lemma proportion}$. Note also that being a proportion $\fat(e_i) \geq 0$, and thus $\sum_{d/2 < e_1,e_2 \leq d} \fat(e_1)\fat(e_2) \geq \sum_{d/2 < e_1,e_2 < d} \fat(e_1) \fat(e_2)$. We obtain
		\[\redIfFat(G) < \frac{\sum_{d/2 < e_1,e_2 < d} 2 \fat(e_1) \fat(e_2) \, q^{-d+1} }{\sum_{d/2 < e_1,e_2 < d} \fat(e_1) \fat(e_2)} = 2q^{-d+1}.\]
\end{proof}

% ------------------------------------------------------------------------

\section*{Acknowledgement}
	The results presented in this paper improve and expand results from the Diplom thesis of the second author, which was submitted at the University of Bayreuth. She thanks her supervisor, Adalbert Kerber, for valuable discussions and academic guidance during that time. The Diplom thesis, on the other hand, developed from her honours thesis submitted at the University of Western Australia under the supervision of the first and third author.
	
	The second author is also indebted to Gerhard Hi{\ss} for numerous discussions and fruitful suggestions which improved the clarity of this work.

% ------------------------------------------------------------------------

\vspace*{1cm}

%----------Author 1
\noindent \textsc{Alice C. Niemeyer} \\
School of Mathematics and Statistics M019\\
The University of Western Australia\\
35 Stirling Highway, Nedlands, WA 6009\\
Australia\\
email: alice.niemeyer@uwa.edu.au \\

%----------Author 2
\noindent \textsc{Sabina B. Pannek}\\
Lehrstuhl D f\"{u}r Mathematik
RWTH Aachen\\
Templergraben 64, 52062 Aachen\\
Germany\\
email: sabina.pannek@gmail.com \\

%----------Author 3
\noindent\textsc{Cheryl E. Praeger}\\
School of Mathematics and Statistics M019\\
The University of Western Australia\\
35 Stirling Highway, Nedlands, WA 6009\\
Australia\\
email: cheryl.praeger@uwa.edu.au

% ------------------------------------------------------------------------
\end{document}